\documentclass[11pt,a4paper]{amsart}

\usepackage{setspace}
\usepackage{amsthm}
\usepackage{amssymb,amsmath,bbm}
\usepackage{amsfonts}
\usepackage[dvips]{graphicx}
\usepackage{pstricks}
\usepackage[ansinew]{inputenc}
\usepackage{enumerate}
\usepackage{multirow}

\newhsbcolor{dar}{1 0.9 0.55}

\newtheorem{theorem}{Theorem}[section]
\newtheorem{corollary}[theorem]{Corollary}
\newtheorem{lemma}[theorem]{Lemma}
\newtheorem{proposition}[theorem]{Proposition}
\newtheorem{problem}[theorem]{Problem}

\theoremstyle{definition}

\newtheorem{assumption}[theorem]{Assumption}

\theoremstyle{remark}

\newcommand{\R}{\mathbb{R}}

\newcommand{\E}{\mathbb{E}}
\newcommand{\p}{\mathbb{P}}

\newcommand{\argmin}{\operatornamewithlimits{argmin}}
\newcommand{\argmax}{\operatornamewithlimits{argmax}}

\begin{document}
\title[Expected Supremum and Singular Stochastic Control]{Expected Supremum Representation of the Value of a Singular Stochastic Control Problem}
\author{Luis H. R. Alvarez E.}\thanks{Turku School of Economics, Department of Accounting and Finance, 20014 University of Turku, Finland,
e-mail:luis.alvarez@tse.fi}
\author{Pekka Matomäki}\thanks{Turku School of Economics, Department of Accounting and Finance, 20014 University of Turku, Finland,
e-mail:pjsila@utu.fi}
\date{June 23, 2015}

\begin{abstract}
We consider the problem of representing the value of singular stochastic control problems of linear diffusions as expected suprema. Setting the value accrued from following a standard reflection policy equal with the expected value of a unknown function at the running supremum of the underlying is shown to result into a functional equation from which the unknown function can be explicitly derived. We also consider the stopping problem associated with the considered singular stochastic control problem and present a similar representation as an expected supremum in terms of characteristics of the control problem.
\end{abstract}
\maketitle
\thispagestyle{empty}
\clearpage \setcounter{page}{1}
\section{Introduction}

Singular stochastic control problems arise quite naturally in economics and finance. Accordingly, a relatively large class of problems utilized in studying the rational management of renewable resources, optimal dividend distribution,  and optimal capital accumulation can be modeled as singular stochastic control problems (see, for example, \cite{Al2011,AlSh98,AlVi,CaSaZa,DeFeMo15,JS,Ko,LES1,LES2,LO,AOks2000}; for an extensive list of references and applications in economics,
see the seminal textbook \cite{Dixit1994}). Given this interest, various techniques for solving singular stochastic control problems have been developed. Especially in the case where the underlying constitutes an It{\^o}-diffusion process variational inequalities have proven to be a very useful tool for characterizing the optimal  policy and its value (\cite{AkMeSu96,BSW,ChMeRo,Kar1,Kar1984}). The advantage of that approach is that it naturally permits the utilization of stochastic calculus and as such applies in a multidimensional and finite horizon setting as well.
Alternatively, given the close connection of the value of the optimal policy with the value of an associated optimal stopping problem the singular stochastic control problem can alternatively be approached by focusing on the associated stopping problem and solving both problems by utilizing probabilistic techniques and relying directly on pathwise properties of the controlled process (cf. \cite{Ba,BaKa,BK,ElKaKa91,ElKarKar94,HaSu1,HaSu2,KS1}).
More recently, it has become apparent that these problems can also be approached by relying directly on the general properties of the running supremum and infimum of stochastic processes and utilizing these processes in the representation of the values of the considered control problems (cf. \cite{BaElKa,BaFo,BaRi,ElKaMe,ElKaFo}). Roughly speaking, these studies delineate a set of general conditions under which a process can be expressed as a functional of the running supremum or infimum of another process and utilize this observation in the characterization of the values of the considered stochastic control problems. This approach provides a general framework for analyzing stochastic control problems since it applies outside the Markovian and infinite horizon setting. Interestingly, this approach extends also to other representation problems and it can be utilized efficiently in the representation of excessive mappings and potentials as expected suprema (cf. \cite{FoKn1,FoKn2,ChSaTa}).

Our objective in this study is to extend the analysis developed in \cite{AlMa15} and  delineate circumstances under which the value of a class of singular stochastic control problems can be represented as an expected running supremum of a function depending on the underlying diffusion process. The main difference to former studies on the subject (e.g. \cite{BaElKa,BaFo,BaRi,ElKaMe,ElKaFo}) is the fact that we are able to derive, in a linear diffusion setting, both the representing function as well as the value of the expected supremum explicitly. These closed form representations allows us to extend our analysis to various ways and establish interesting connection between different representations.

Instead of determining the optimal policy and its value directly, we first consider the value of a standard reflection policy at an arbitrary constant boundary. We will introduce the problem in exact terms in the next section, but to get a general idea of the chosen approach, consider the value of a singular control (a {\em local time push}) characterized by downward reflection at an arbitrary constant threshold $y$:
\begin{align*}
H_y(x)=\E_x\left[\int_0^{\tau_0^Z} e^{-rs}\left(\alpha(X_s^{Z^y})-(R_r\pi)'(X_s^{Z^y})\right)dZ^y_s\right],
\end{align*}
where $X^{Z^y}$ denotes the reflected diffusion, $Z^y$ denotes the control policy, $\tau_0^Z=\inf\{t\geq 0: X_t^Z\leq 0\}$ denotes the exogenous liquidation date, and $(R_r\pi)(x)$ and $\alpha(x)$ are sufficiently nice functions. The principal objective of our study is to explicitly determine the representing function $f(x)$ for which
\begin{align*}
\E_x\left[f\left(\sup_{0\leq t\leq T_r} X_t\right)\mathbbm{1}_{(y,\infty)}\left(\sup_{0\leq t\leq T_r} X_t\right)\right]=H_y(x),
\end{align*}
where $T_r\sim \text{Exp}(r)$ is an exponentially distributed random time independent of the underlying diffusion $X$ killed at $0$ in the absence of interventions (see Theorems \ref{thm1} and \ref{singesitys1}). Computing the left hand side of the equation above by utilizing the explicitly known probability distribution of the supremum attained at an exponentially distributed random time and requiring equality of these two functional forms then results into an identity from which the unknown representing function can be explicitly determined.

Somewhat surprisingly, and in contrast with representation results of optimal stopping problems, the representing function $f$ is shown to vanish at an arbitrary reflection boundary $y$ --- independently of whether the boundary is optimal or not. Hence, the condition resulting to the validity of the standard smooth fit principle in an optimal stopping setting is not sufficient in the optimal reflection case. We demonstrate that in the singular stochastic control case optimality is attained at the threshold where the derivative $f'$ vanishes. As intuitively is clear, this condition is shown to coincide with the standard $C^2$-smoothness requirement across the optimal reflection boundary; a condition which typically arises in singular stochastic control problems of linear diffusions.

Given that the derivative of the value of a singular stochastic control problem usually coincides with the value of an associated optimal stopping problem, we also analyze the optimal stopping problem associated with the marginal value of the optimal reflection policy. Furthermore, we derive a similar supremum representation for the associated stopping problem under some circumstances as well. We find that the representing functions needed for the representations mentioned above are functionally dependent on each other so that the one associated with the stopping problem can be seen as multiple of the derivative of the representing function associated with the singular stochastic control problem. For the sake of completeness, we also analyze the associated nonlinear optimal stopping problem arising in the analysis of Gittins indices and characterize this value in terms of the derived functional forms.

The contents of this study are as follows. In section 2 we formulate the considered singular stochastic control problem and  characterize the underlying dynamics as well as the key functional forms needed later in the analysis. Our main findings on the considered representation problem are then summarized in 3. These findings are then illustrated numerically in an explicitly parameterized model in section 4. Finally, section 5 concludes our study.

\section{Singular Stochastic Control Problem}

Our objective is to characterize the value of a relatively broad class of singular stochastic control problems as expected suprema of a measurable function satisfying a set of regularity and monotonicity conditions stated later in this study. To this end, let  $(\Omega,\p,\{\mathcal{F}_t\}_{t\geq0},\mathcal{F})$ be a complete filtered probability space satisfying the usual conditions and assume that the underlying diffusion
process evolves on $\mathcal{I}=(0,b)\subseteq \R_+$ according to the dynamics described by the generalized stochastic
differential equation
\begin{align}
dX_t^Z = \mu(X_t^Z)dt + \sigma(X_t^Z)dW_t -  dZ_t,\quad
X_0^Z=x\in \mathcal{I},\label{singcontsde}
\end{align}
where the drift coefficient $\mu:\mathcal{I}\mapsto \mathbb{R}$ and the volatility coefficient $\sigma:\mathcal{I}\mapsto \mathbb{R}_+$ are assumed to be continuous and to satisfy the conditions
$\sigma(x)>0$ for all $x\in
\mathcal{I}$ and
$$
\int_{x-\varepsilon}^{x+\varepsilon}\frac{1 + |\mu(y)|}{\sigma^2(y)}dy < \infty
$$
for some $\varepsilon > 0$ and all $x\in \mathcal{I}$. These
conditions guarantee the existence of a weak solution for the
stochastic differential equation \eqref{singcontsde} (cf. \cite{Karatzas88}, pp.
342--353). The control policy $Z$ is assumed to satisfy $Z\in \mathcal{S}$, where $\mathcal{S}$ denotes the class of non-negative, non-decreasing,
right-continuous, and $\{{\mathcal F}_{t}\}$-adapted processes. In what follows, we will denote by $X$ the diffusion $X^Z$, when the control policy $Z$ does not directly affect the underlying randomly fluctuating process driving the stochasticity of the model. We will also assume throughout this study that the upper boundary $b$ is natural for the underlying diffusion $X$ and that the lower boundary is either natural, exit, entrance, or regular for $X$.

Given these dynamics, our objective is to
analyze the singular stochastic control problem
\begin{align}
V(x) = \sup_{Z\in \mathcal{S}} \mathbb{E}_x\int_0^{\tau_0^Z}e^{-rs}
\left(\pi(X_s^Z)ds + \alpha(X_s^Z) dZ_s\right), \label{singcont}
\end{align}
where $\pi:\mathcal{I}\mapsto \mathbb{R}$ is a known function measuring the revenue flow accrued from continuing operation and $\alpha:\mathcal{I}\mapsto\mathbb{R}_+$ denotes the state dependent marginal return associated with utilizing the control policy. We will later impose tighter regularity conditions on these mappings in order to be able to characterize the values and their representations as expected suprema explicitly.

As usually, we denote by $$\mathcal{A} = \frac{1}{2}\sigma^2(x)\frac{d^2}{dx^2}+\mu(x)\frac{d}{dx}$$ the differential operator representing the infinitesimal generator of the underlying diffusion $X$. As is known from the classical
theory on linear diffusions,
there are two linearly independent {\it fundamental solutions}
$\psi(x)$ and $\varphi(x)$ satisfying a set of appropriate boundary
conditions based on the boundary behavior of the process $X$ and
spanning the set of solutions of the ordinary second order differential equation
$(\mathcal{G}_ru)(x)=0$, where $\mathcal{G}_r=\mathcal{A}-r$ denotes the differential operator associated with the diffusion $X$ killed at the constant rate $r$. Moreover, $\psi'(x)\varphi(x) - \varphi'(x)\psi(x) = BS'(x),$
where $B>0$ denotes the constant Wronskian of the fundamental
solutions and  (for a comprehensive characterization of the fundamental solutions, see \cite{BorSal02}, pp. 18-19)
$$
S'(x)=\exp\left(-\int^x\frac{2\mu(t)}{\sigma^2(t)}dt\right)
$$
denotes the density of the scale function of $X$. In what follows, we will denote by $\hat{\psi}_z(x)= \psi(x) - \frac{\psi(z)}{\varphi(z)}\varphi(x)$ the increasing fundamental solution associated with the diffusion $X$ killed at the boundary $z\in \mathcal{I}$. It is clear that (cf. \cite{BorSal02}, p. 19)
$$
\lim_{z\downarrow 0}\hat{\psi}_z(x)=\begin{cases}
\psi(x) - \frac{\psi(0)}{\varphi(0)}\varphi(x) & \textrm{if }0\textrm{ is regular for }X\\
\psi(x) & \textrm{otherwise}.
\end{cases}
$$

Given the fundamental solutions characterized above, let $u(x)=c_1\psi(x) + c_2\varphi(x), c_1,c_2\in \mathbb{R}$ be an arbitrary twice continuously differentiable $r$-harmonic function and define for sufficiently smooth mappings $g:\mathcal{I}\mapsto\R$ the functional
\begin{align*}
(L_u g)(x) = g(x) \frac{u'(x)}{S'(x)}-\frac{g'(x)}{S'(x)}u(x)
\end{align*}
associated with the representing measure for $r$-excessive functions (cf. \cite{Salminen1985}). Noticing that if $g$ is twice continuously differentiable, then
$$
(L_u g)'(x)=-(\mathcal{G}_rg)(x)u(x)m'(x)
$$
implying that
\begin{align}\label{exrep}
(L_u g)(z)-(L_u g)(y)=\int_z^y(\mathcal{G}_rg)(t)u(t)m'(t)dt
\end{align}
for any $a<z<y<b$.
Finally, we denote by $\mathcal{L}_r^1(\mathcal{I})$ the class of measurable functions $f:\mathcal{I}\mapsto\R_+$ satisfying for all $x\in \mathcal{I}$ the integrability condition
$$
\mathbb{E}_x\int_0^{\tau_0} e^{-rs}|f(X_s)|ds<\infty,
$$
where $\tau_0=\inf\{t\geq 0:X_t=0\}$ denotes the first hitting time to $0$. As is known from the literature on linear diffusions, the expected cumulative present value of a function $f\in \mathcal{L}_r^1(\mathcal{I})$, that is,
$$
(R_rf)(x)=\mathbb{E}_x\int_0^{\tau_0} e^{-rs}f(X_s)ds
$$
can be expressed as
\begin{align}
(R_rf)(x)= B^{-1}\varphi(x)\int_0^x \hat{\psi}_0(y)f(y)m'(y)dy+B^{-1}\hat{\psi}_0(x)\int_x^b \varphi(y)f(y)m'(y)dy,\label{Green}
\end{align}
where
$
m'(x) = 2/(\sigma^{2}(x)S'(x))
$
denotes the density of the speed measure $m$ of $X$.

In order to guarantee that the considered singular stochastic control problem \eqref{singcont} is well-defined we first have to make some assumptions on the key functional forms. We first make the following assumptions:
\begin{assumption}\label{1assumptions}
We assume that throughout this section the following conditions are met:
\begin{itemize}
  \item[(a)] $\pi\in\mathcal{L}_r^1(\mathcal{I})$ and $\pi\in C^1(\mathcal{I})$,
  \item[(b)] $\alpha\in C^1(\mathcal{I})$, $\alpha(x)$ attains a unique global maximum at $x_\alpha\in[0,b)$ and is non-increasing on $(x_\alpha,b)$, and
  \item[(c)] $\lim_{x\uparrow b}\Lambda(x)/\hat{\psi}_0(x)=0$, where
    $$
\Lambda(x):=\int_0^x\alpha(z)dz.
$$
\end{itemize}
\end{assumption}
Assumption \ref{1assumptions}(a) essentially guarantees that the value accrued from leaving the opportunity to use the control policy unutilized is finite. Assumption \ref{1assumptions}(b), in turn, guarantees that while the marginal return associated with the control policy is decreasing as a function of the state (i.e. the smaller the reserves get, the higher the price of a remaining unit of stock becomes) its cumulative value cannot explode. In this way it creates incentives to maintain the reserves unutilized at least at low levels. Finally, assumption \ref{1assumptions}(c) guarantees that even though the cumulative marginal returns associated with the instantaneous application of the control policy are increasing, they do not explode at a rate which would result into the optimality of a single "take the money and run"- policy exerted at an arbitrarily high state (i.e. instantaneous depletion of reserves as soon as the underlying exceeds a beforehand determined threshold).

It is at this point worth commenting shortly our assumptions on the function $\alpha(x)$ characterizing the marginal return associated with the utilization of the control policy. As we will later establish in our paper, the existence and uniqueness of an optimal exercise threshold at which the underlying diffusion should be reflected downwards does not, per se, impose monotonicity requirements on $\alpha(x)$. Hence, the existence of such a barrier can be established even in cases where $\alpha(x)$ does not satisfy Assumption \ref{1assumptions}(b). Essentially, Assumption \ref{1assumptions}(b) is needed in order to guarantee that the initial optimal policy cannot be of the instantaneous shifting down to the optimal boundary type and that it has to be either of a standard instantaneous "local time" reflection type (cf. \cite{Shreve84}) or a {\em chattering down to the optimal threshold}-type (cf. \cite{Al00}). The latter of these policies does not belong into the set of admissible policies and, thus, it must be interpreted as a limit of admissible distribution strategies. In cases where $\alpha(x)$ is, for example, increasing the potentially optimal policies (at date $0$) are either of the reflection type or of the immediate shifting down to the optimal boundary-type. The reason for this argument is that if $\alpha(x)$ is increasing then we naturally have $\Lambda(x)-\Lambda(y)< \alpha(x)(x-y)$ which shows that the gradual distribution of reserves above an fixed potentially optimal boundary is suboptimal. Finally, in a general setting where $\alpha(x)$ in not monotonic, the optimal policy can be characterized as a combination of chattering and instantaneous shifting.

Define the mapping $A:\mathcal{I}\mapsto \mathbb{R}$ measuring the option value of depleting the reserves instantaneously instead of leaving the reserves uncontrolled by $A(x)=\Lambda(x)-(R_r\pi)(x)$. Let us now consider the parameterized family of continuously differentiable functions $F_y(x)=(R_r\pi)(x)+H_y(x)$, where (cf. \cite{AlVi})
\begin{align}\label{family1}
H_y(x) = \begin{cases}
A(x)- A(y) + A'(y)\frac{\hat{\psi}_0(y)}{\hat{\psi}_0'(y)}&x\in [y,b)\\
A'(y)\frac{\hat{\psi}_0(x)}{\hat{\psi}_0'(y)} &x\in(0,y).
\end{cases}
\end{align}
As is known from the literature on singular stochastic control, $F_y(x)$ represents the value accrued from following a reflecting strategy at a given exogenous boundary $y\in \mathcal{I}$ (cf. \cite{Shreve84}). Hence, the value $F_y(x)$ constitutes a decomposition where the first cumulative part $(R_r\pi)(x)$ measures the value of continuing operation with the prevailing stock $x$ while the second part $H_y(x)$ captures the excess value which can be accrued by utilizing a local time type control policy at a given exogenous boundary $y$. It is worth emphasizing that the function $H_y(x)$ is also closely related to the value of the associated singular stochastic control problem
\begin{align}\label{singular2}
H(x)=\sup_{Z\in \mathcal{S}}\E_x\int_0^{\tau_0^Z}e^{-rs}A'(X_s^Z)dZ_s.
\end{align}
To see that this is indeed the case in the present situation, consider only those admissible policies $Z\in \mathcal{S}$ which maintain the underlying diffusion in an arbitrary bounded set $(0,y]$. Utilizing in that case the generalized It{\^o}-Döblin theorem (Doléans-Dade-Meyer formula) to the function $(t,x)\mapsto e^{-rt}(R_r\pi)(x)$ yields (cf. Theorem 32 in \cite{Protter05}; see also Lemma 3.1 in \cite{Ma12})
\begin{align*}
\E_x\left[e^{-r T_n^Z}(R_r\pi)(X_{T_n^Z}^Z)\right] = (R_r\pi)(x)-\E_x\int_0^{T_n^Z}e^{-rs}\pi(X_s^Z)ds +\E_x\int_0^{T_n^Z}e^{-rs}(R_r\pi)'(X_s^Z)dZ_s,
\end{align*}
where $T_n^Z$ is a sequence of almost surely finite stopping times converging to $\tau_0^Z$ as $n\rightarrow\infty$. Letting $n\rightarrow\infty$ now yields
$$
\E_x\int_0^{\tau_0^Z}e^{-rs}\pi(X_s^Z)ds=(R_r\pi)(x)+\E_x\int_0^{\tau_0^Z}e^{-rs}(R_r\pi)'(X_s^Z)dZ_s
$$
for any admissible policy $Z\in \mathcal{S}$ maintaining the underlying diffusion in an arbitrary bounded set $(0,y]$. This shows how the singular stochastic control problems \eqref{singcont} and \eqref{singular2} are related.

\section{Representation of Values as Expected Suprema}

Let $M_t=\sup\{X_s;s\leq t\wedge\tau_0\}$ denote the running supremum of $\tilde{X}=\{X_t;t< \tau_0\}$ and assume that $T\sim\exp(r)$ is an exponentially distributed random time
which is independent of the underlying diffusion.
Our objective is to now consider the following identification problem:
\begin{problem}
Does there exist a function $\hat{f}(x)=f(x)\mathbbm{1}_{[y,b)}(x)$ such that for all $x\in\mathcal{I}$ we would have
$H_y(x)=\mathbb{E}_x[\hat{f}(M_T)]<\infty$ for all $x\in\mathcal{I}$. Under which conditions we have $H_y(x)=\mathbb{E}_x[\sup\{\hat{f}(X_t);t\leq T\}]<\infty$ for all $x\in \mathcal{I}$.
\end{problem}
Our objective is in what follows to characterize the function $\hat{f}$ explicitly. To this end, we first observe that since (cf. p. 26 in \cite{BorSal02})
$$
\mathbb{P}_x\left[M_T\geq y\right] = \frac{\hat{\psi}_0(x)}{\hat{\psi}_0(y)},
$$
we have
\begin{align}\label{esitys1}
\mathbb{E}_x[\hat{f}(M_T)] = \hat{\psi}_0(x)\int_{x\lor y}^bf(z)\frac{\hat{\psi}_0'(z)}{\hat{\psi}_0^2(z)}dz.
\end{align}
Thus, identity $H_y(x)=\mathbb{E}_x[\hat{f}(M_T)]$ holds if
\begin{align}\label{Volterra}
\int_{x\lor y}^bf(z)\frac{\hat{\psi}_0'(z)}{\hat{\psi}_0^2(z)}dz= \frac{H_y(x)}{\hat{\psi}_0(x)}=\begin{cases}
\frac{A(x)- A(y) + A'(y)\frac{\hat{\psi}_0(y)}{\hat{\psi}_0'(y)}}{\hat{\psi}_0(x)}&x\in [y,b)\\
\frac{A'(y)}{\hat{\psi}_0'(y)} &x\in(0,y).
\end{cases}
\end{align}
It is now a straightforward exercise to show that this identity holds for $x\geq y$ if the function $f(x)$ is chosen according to the identity (cf. Section 3.2 in \cite{BaFo})
\begin{align}
f(x)=A(x)-\frac{A'(x)}{\hat{\psi}_0'(x)}\hat{\psi}_0(x)-\left(A(y)-\frac{A'(y)}{\hat{\psi}_0'(y)}\hat{\psi}_0(y)\right)=
\frac{S'(x)}{\hat{\psi}_0'(x)}(L_{\hat{\psi}_0}A)(x)-\frac{S'(y)}{\hat{\psi}_0'(y)}(L_{\hat{\psi}_0}A)(y).\label{incsing}
\end{align}
We have, thus, been able to show the following representation result.
\begin{theorem}\label{thm1}
Assume that $f$ is defined as in \eqref{incsing}. Then, $H_y(x)=\mathbb{E}_x[f(M_T)\mathbbm{1}_{[y,b)}(M_T)]$ for all $x\in \mathcal{I}$.
\end{theorem}
According to Theorem \ref{thm1}, the excess value which can be accrued by utilizing a local time type control policy at a given exogenous boundary $y$
can be expressed as the expected value of a function $\hat{f}$ at the running maximum attained at an exponentially distributed random date. However,
despite this identity, Theorem \ref{thm1} does not make statements on the potential optimality of the considered control policy nor does it characterize circumstances under
which the value $H_y(x)$ could be expressed as an expected supremum. It is clear by definition of the function $\hat{f}$ that in order to represent the value
as an expected supremum, the function $\hat{f}$ needs to be nonnegative and nondecreasing. Since
$f(y+)=0$ by definition, we notice that $\hat{f}(x)$ is nonnegative and nondecreasing provided that
\begin{align}
f'(x)=-\frac{d}{dx}\left[\frac{A'(x)}{\hat{\psi}_0'(x)}\right]\hat{\psi}_0(x)\geq 0\label{monotonicity}
\end{align}
for all $x\geq y$. We can now prove our first main representation theorem related to the singular stochastic control problem \eqref{singcont}.
\begin{theorem}\label{singesitys1}
Assume that the function $A'(x)/\hat{\psi}_0'(x)$ attains a unique global maximum at the threshold $y^\ast=\argmax\left\{A'(x)/\hat{\psi}_0'(x)\right\}$ and that $A'(x)/\hat{\psi}_0'(x)$ is non-increasing on $[y^\ast,b)$. If $f$ is defined as in \eqref{incsing}, then
\begin{align}
H_y(x)=\mathbb{E}_x\left[\sup\left\{f(X_t)\mathbbm{1}_{[y,b)}(X_t);t\leq T\right\}\right]\label{rep1}
\end{align}
is $r$-excessive for $X$ for all $y\geq y^\ast$ and $V(x)=F_{y^\ast}(x)$.
\end{theorem}
\begin{proof}
The assumed monotonicity of $A'(x)/\hat{\psi}_0'(x)$ on $[y^\ast,b)$ guarantees that
$$
\frac{d}{dx}\frac{A'(x)}{\hat{\psi}_0'(x)} \leq 0
$$
for all $x\geq y^\ast$. Condition $f(y+)=0$ and inequality \eqref{monotonicity} now guarantees that $\hat{f}$ is continuous, nonnegative, and monotonically increasing for all $y\geq y^\ast$ and, therefore,
that $H_y(x)=\mathbb{E}_x[\sup\{\hat{f}(X_t);t\leq T\}]$ in that case.  The $r$-excessivity of $H_y(x)$ then follows from Proposition 2.1 in \cite{FoKn1}.
The identity $V(x)=F_{y^\ast}(x)$ follows from Theorem 1 in \cite{Al00}.
\end{proof}
Theorem \ref{singesitys1} states in terms of the ratio $A'(x)/\hat{\psi}_0'(x)$ a set of sufficient conditions under which the value of a standard reflection policy at a given fixed threshold can be expressed as an expected supremum. Given that the value represents the expected cumulative returns obtained by following a
reflection policy, it is clear that the highest value which can be attained by following such a policy is the one where the reflection policy is exerted at the threshold maximizing the ratio $A'(x)/\hat{\psi}_0'(x)$. As was pointed out earlier in our paper, Theorem \ref{singesitys1} does not impose strong monotonicity requirements on the marginal return $\alpha(x)$. As is clear from Theorem \ref{singesitys1}, the validity of the considered representation is guaranteed as long as the ratio $A'(x)/\hat{\psi}_0'(x)$ attains a unique global maximum above which it is non-increasing.

A set of sufficient condition characterizing  circumstances under which the requirements of Theorem \ref{singesitys1} are met are now summarized in the following.
\begin{theorem}\label{singesitys2}
Assume that the function $(\mathcal{G}_rA)(x)$ satisfies the following conditions:
\begin{itemize}
  \item[(i)] there exists a $\hat{x}=\argmax\{(\mathcal{G}_rA)(x)\}\in \mathcal{I}$ so that $(\mathcal{G}_rA)(x)$ is nondecreasing on $(0,\hat{x})$  and non-increasing on $(\hat{x},b)$
  \item[(ii)] $\lim_{x\downarrow 0}(\mathcal{G}_rA)(x)\geq 0$ if $0$ is unattainable for $X$, $\lim_{x\downarrow 0}(\mathcal{G}_rA)(x)> 0$ if $0$ is attainable for $X$, and $\lim_{x\uparrow b}(\mathcal{G}_rA)(x)<-\varepsilon$, where $\varepsilon>0$.
\end{itemize}
Then, there exists a unique threshold
$$
y^\ast=\argmax\left\{\frac{A'(x)}{\hat{\psi}_0'(x)}\right\}\in(\hat{x},x_0)
$$
where $x_0=\inf\{x>\hat{x}: (\mathcal{G}_rA)(x)=0\}$, satisfying the ordinary first order optimality condition
\begin{align}
r\int_0^{y^\ast}\hat{\psi}_0(t)(\mathcal{G}_rA)(t)m'(t)dt-(\mathcal{G}_rA)(y^\ast)\frac{\hat{\psi}_0'(y^\ast)}{S'(y^\ast)}=0.\label{foc1}
\end{align}
Especially, if $f$ is defined as in \eqref{incsing}, then the representation \eqref{rep1}
is $r$-excessive for $X$ for all $y\geq y^\ast$ and $V(x)=F_{y^\ast}(x)$, where
\begin{align}
F_{y^\ast}(x)=(R_r\pi)(x) + \frac{(\mathcal{G}_rA)(y^\ast)}{r}  - \mathbb{E}_x\left[\inf\left\{\frac{S'(X_t\lor y^\ast)}{\hat{\psi}_0'(X_t\lor y^\ast)}\int_0^{X_t\lor y^\ast}\hat{\psi}_0(s)(\mathcal{G}_rA)(s)m'(s)ds;t\leq T\right\}\right].\label{rep2}
\end{align}
\end{theorem}
\begin{proof}
As is well-known from the literature on diffusion processes, the resolvent $R_r$ and the infinitesimal generator $\mathcal{A}-r$ are inverse operators (see e.g. III.4 in \cite{RogWil00}). Hence, under our assumptions we naturally have $A(x)=-\left(R_r(\mathcal{G}_rA)\right)(x)$.  By applying the representation \eqref{Green}, we observe that the existence and uniqueness of the minimizing threshold $y^\ast$ now follows from the identity
\begin{align*}
\frac{d}{dx}\frac{A'(x)}{\hat{\psi}_0'(x)} = \frac{2S'(x)}{\sigma^2(x)\hat{\psi}_0'^{2}(x)}\left[r\int_0^x\hat{\psi}_0(t)(\mathcal{G}_rA)(t)m'(t)dt-(\mathcal{G}_rA)(x)\frac{\hat{\psi}_0'(x)}{S'(x)}\right]
\end{align*}
in a completely analogous way with the proof of Theorem 2 in \cite{Al00}.
Utilizing this identity now proves that
$$
\frac{d}{dx}\frac{A'(x)}{\hat{\psi}_0'(x)} = -\frac{d}{dx}\frac{\left(R_r(\mathcal{G}_rA)\right)'(x)}{\hat{\psi}_0'(x)} \leq 0
$$
for all $x\geq y^\ast$. Condition $f(y+)=0$ and inequality \eqref{monotonicity} now guarantees that $\hat{f}$ is continuous, nonnegative, and monotonically increasing for all $y\geq y^\ast$ and, therefore,
that $H_y(x)=\mathbb{E}_x[\sup\{\hat{f}(X_t);t\leq T\}]$ in that case.  The $r$-excessivity of $H_y(x)$ then follows from Proposition 2.1 in \cite{FoKn1}.
Finally, applying representation \eqref{Green} to the present problem yields
$$
A(x)-\frac{A'(x)}{\psi'(x)}\psi(x)=-\frac{S'(x)}{\hat{\psi}_0'(x)}\int_0^x\hat{\psi}_0(t)(\mathcal{G}_rA)(t)m'(t)dt.
$$
Thus, we observe by setting $y=y^\ast$ that
$$
\hat{f}(x)=\frac{(\mathcal{G}_rA)(y^\ast)}{r}-\frac{S'(x\lor y^\ast)}{\hat{\psi}_0'(x\lor y^\ast)}\int_0^{x\lor y^\ast}\hat{\psi}_0(t)(\mathcal{G}_rA)(t)m'(t)dt
$$
demonstrating the validity of the proposed representation \eqref{rep2}.
\end{proof}
Theorem \ref{singesitys2} presents a set of easily verifiable sufficient conditions under which the value of the considered singular stochastic control problem admits a
representation as an expected supremum. Since this representation is closely associated with the excess value which can be accrued by controlling the underlying dynamics, we find
that \eqref{rep2} provides a supremum interpretation of the maximal yield accrued from controlling the reserves in an optimal fashion. An interesting corollary of Theorem \ref{singesitys2} is summarized in the following.
\begin{corollary}\label{signcor}
Assume that the conditions of Theorem \ref{singesitys2} are satisfied and that the  representing function $f$ is defined as in \eqref{incsing}. Then, $f(x)\geq 0$ for all $x\in \mathcal{I}$ only if $y=y^\ast$.
\end{corollary}
\begin{proof}
Noticing that $y^\ast=\argmin\left\{A(x)-\frac{A'(x)}{\hat{\psi}_0'(x)}\hat{\psi}_0(x)\right\}$ demonstrates that the alleged result follows from \eqref{incsing}.
\end{proof}
According to Corollary \ref{signcor}, the representing function $f$ defined in \eqref{incsing} can be nonnegative on the entire state space $\mathcal{I}$ only if the constant boundary $y$ is chosen optimally. Otherwise, the set where $f$ is nonpositive is nonempty. This is an interesting observation, since even though the representing function $f$ is constructed so as to guarantee that condition $f(y)=0$ is always met, it is not necessarily nonnegative on $\mathcal{I}$.

Theorems \ref{thm1} and \ref{singesitys2} characterize the auxiliary function $f$ needed for the representation of the value as an expected supremum explicitly by deriving first the value of the optimal strategy and then solving the Volterra equation \eqref{Volterra}. It is, naturally, of interest to ask if the reverse is also true. More precisely, can the value and optimal strategy be directly characterized by a given $f$ possessing a set of suitable properties? A set of circumstances under which the answer to this question is positive are summarized in the following proposition.
\begin{proposition}\label{prop y}
Assume that $\hat{\psi}_0(x)$ is convex and let $y^\ast:=\inf\{y\in \mathcal{I}: \tilde{f}(y)\geq 0\}$, where $\tilde{f}:\mathcal{I}\mapsto \mathbb{R}$ is defined by
$$
\tilde{f}(x) = A'(x)-\frac{A''(x)}{\hat{\psi}_0''(x)}\hat{\psi}_0'(x).
$$
Assume also that $\tilde{f}(x)$ is positive for all $x>y^*$. Then, $y^\ast=\argmax\left\{A'(y)/\psi'(y)\right\}$ and the conclusions of Theorem \ref{singesitys1} hold.
\end{proposition}
\begin{proof}
The identity $y^\ast=\argmax\left\{A'(y)/\psi'(y)\right\}$ follows directly from
\begin{align*}
f'(x)=-\frac{d}{dx}\left(\frac{A'(x)}{\psi'(x)}\right)\psi(x)=\tilde{f}(x)\frac{\psi''(x)}{\psi'(x)^2}\psi(x).
\end{align*}
Furthermore, since $\tilde{f}(x)> 0$ on $(y^\ast, b)$ by assumption, we notice that $f'(x)>0$ for all $x>y^\ast$. The alleged claim now follows from Theorem \ref{singesitys1}.
\end{proof}
Proposition \ref{prop y} delineates circumstances under which the optimal reflection threshold can be identified as the smallest state at which an associated function $\tilde{f}$ remains positive.
As we will see in what follows, the function $\tilde{f}$ is related to a stopping problem characterizing the  marginal value of the optimal policy.

\section{Associated Optimal Stopping Problem}
Having analyzed the representation of the value of the optimal singular policy as an expected supremum, we now plan to consider the special case where
the marginal value $V'(x)$ of the considered singular stochastic control problem \eqref{singcont} constitutes the solution of an associated stopping problem and extend our representation to that setting as well. More precisely, we now consider the parameterized family of continuous functions $F_y'(x)=(R_r\pi)'(x)+H_y'(x)$, where
\begin{align}\label{family2}
H_y'(x) = \begin{cases}
A'(x)&x\in [y,b)\\
A'(y)\frac{\hat{\psi}_0'(x)}{\hat{\psi}_0'(y)} &x\in(0,y).
\end{cases}
\end{align}
Before proceeding we establish the following auxiliary result.
\begin{lemma}\label{apua1}
Assume that conditions of Theorem \ref{singesitys2} are satisfied. In that case the marginal value $V'(x)=(R_r\pi)'(x)+H_{y^\ast}'(x)$ of the optimal policy can be expressed as
\begin{align}\label{family3}
V'(x)= (R_r\pi)'(x)+\inf_{y\geq x}\left\{\frac{A'(y)}{\hat{\psi}_0'(y)}\right\}\hat{\psi}_0'(x) = (R_r\pi)'(x)+\frac{A'(y^\ast\lor x)}{\hat{\psi}_0'(y^\ast\lor x)}\hat{\psi}_0'(x).
\end{align}
Moreover, the mapping $(\mathcal{G}_rV)(x)+\pi(x)$ is non-positive and non-increasing on $\mathcal{I}$.
\end{lemma}
\begin{proof}
The representation \eqref{family3} is a direct consequence of the findings of Theorem \ref{singesitys2}. Moreover, since $V(x)=F_{y^\ast}(x) = (R_r\pi)(x)+H_{y^\ast}(x)$ satisfies the identity $(\mathcal{G}_rF_{y^\ast})(x)+\pi(x)=((\mathcal{G}_rA)(x)-(\mathcal{G}_rA)(y^\ast))\mathbbm{1}_{[y^\ast,b)}(x)$ for all $x\in \mathcal{I}$ and $(\mathcal{G}_rA)(x)$ is decreasing on
$[y^\ast,b)$, we find that $(\mathcal{G}_rV)(x)+\pi(x)$ is non-positive and non-increasing on $\mathcal{I}$.
\end{proof}
Lemma \ref{apua1} demonstrates that the marginal value of the optimal singular stochastic control policy constitutes the solution of an associated nonlinear programming problem. As we will observe later in this study, combining this property with the monotonicity of the mapping $(\mathcal{G}_rV)(x)+\pi(x)$ can be utilized in the determination of situations under which the marginal value $V'(x)$ constitutes the value of an associated optimal stopping problem.
In order to accomplish this task, we first have to make some extra assumptions on the infinitesimal characteristics of the underlying process. These extra assumptions are listed in the following.
\begin{assumption}\label{oletukset}
Assume that the following conditions are met:
\begin{itemize}
  \item[(A)] the infinitesimal drift $\mu(x)$ and volatility coefficient $\sigma(x)$ are continuously differentiable on $\mathcal{I}$ and
$$
\int_{x-\varepsilon}^{x+\varepsilon}\frac{1 + |\mu(y)+\sigma'(y)\sigma(y)|}{\sigma^2(y)}dy < \infty
$$
for some $\varepsilon > 0$ and all $x\in \mathcal{I}$
\item[(B)]  $r>\mu'(x)$ for all $x\in \mathcal{I}$, that is, the rate at which a unit of account appreciates is decreasing.
\end{itemize}
\end{assumption}

Assumption \ref{oletukset}(A) guarantee the existence of a weak solution for the
stochastic differential equation (cf. \cite{Karatzas88}, pp.
342--353)
\begin{align}\label{sde2}
d\hat{X}_t = \left(\mu(\hat{X}_t)+\sigma'(\hat{X}_t)\sigma(\hat{X}_t)\right)dt + \sigma(\hat{X}_t)dW_t,\quad \hat{X}_0=x.
\end{align}
It is clear that the density of the scale of $\hat{X}$ is related to the density of the scale of $X$ through the identity $\hat{S}'(x)=S'(x)/\sigma^2(x)$. The density of the speed measure of $\hat{X}$ reads, in turn, as $\hat{m}'(x)=2/S'(x)$.

Given the process $\hat{X}$ characterized by the stochastic differential equation \eqref{sde2} and the marginal value \eqref{family3} of the optimal policy, we now consider the optimal stopping problem
\begin{align}\label{family2b}
\hat{V}(x) = (R_r\pi)'(x) + \sup_{\tau} \E_x\left[e^{-\int_0^{\tau\wedge \hat{\tau}_0}\rho(\hat{X}_s)ds}A'(\hat{X}_{\tau\wedge\hat{\tau}_0})\right],
\end{align}
where $\hat{\tau}_0 = \inf\{t\geq 0:\hat{X}_t=0\}$ and $\rho(x)=r-\mu'(x)$. Given that the marginal expected cumulative present value $(R_r\pi)'(x)$ is independent of the stopping strategy, we can naturally focus on the stopping problem
\begin{align}\label{family2c}
\hat{H}(x) = \sup_{\tau} \E_x\left[e^{-\int_0^{\tau\wedge \hat{\tau}_0}\rho(\hat{X}_s)ds}A'(\hat{X}_{\tau\wedge\hat{\tau}_0})\right]
\end{align}
and compare it with  the marginal value of \eqref{singular2}.

We immediately obtain the following useful result.
\begin{lemma}\label{apulemma}
Assume that the conditions of Theorem \ref{singesitys2} and Assumptions \ref{oletukset} are satisfied. Then, $H'(x)=H_{y^\ast}'(x) \geq \hat{H}(x)$ for  all $x\in \mathcal{I}$. Moreover, if $0$ is unattainable for $X$, $\lim_{x\downarrow 0}\varphi'(x)=-\infty$ and $\lim_{x\uparrow b}\psi'(x)=\infty$, then $H'(x)=H_{y^\ast}'(x) = \hat{H}(x)$ and $\tau^\ast=\inf\{t\geq 0: \hat{X}_t\geq y^\ast\}$ is an optimal stopping time.
\end{lemma}
\begin{proof}
We notice that Theorem \ref{singesitys1} guarantees that $H(x)=H_{y^\ast}(x)$. Utilizing this identity and the known properties of the value $H_{y^\ast}(x)$ first implies that $H_{y^\ast}'(x) \geq A'(x)$ for all $x\in \mathcal{I}$. Second, since $H_{y^\ast}(x)$ satisfies the identity $(\mathcal{G}_rH_{y^\ast})(x)=0$ for all $x\in (0,y^\ast)$, we find by ordinary differentiation $(\mathcal{G}_rH_{y^\ast})'(x)=0$ for all $x\in (0,y^\ast)$. Third, since $(\mathcal{G}_rH_{y^\ast})(x)=(\mathcal{G}_rA)(x)+r\left(A(y^\ast)-A'(y^\ast)\frac{\hat{\psi}_0(y^\ast)}{\hat{\psi}_0'(y^\ast)}\right)$ for all $x\in [y^\ast,b)$, we find by ordinary differentiation $(\mathcal{G}_rH_{y^\ast})'(x)=(\mathcal{G}_rA)'(x)\leq 0$ for all $x\in (y^\ast,b)$. Consequently, we observe that $(\mathcal{G}_rH_{y^\ast})'(x)\leq 0$ for all $x\in \mathcal{I}\setminus\{y^\ast\}$ and $|H_{y^\ast}'''(y^\ast\pm)|<\infty$. Since
$$
(\mathcal{G}_rH_{y^\ast})'(x)= \frac{1}{2}\sigma^2(x)H_{y^\ast}'''(x)+(\mu(x)+\sigma(x)\sigma'(x))H_{y^\ast}''(x)-\rho(x)H_{y^\ast}(x)
$$
we find by applying the It{\^o}-Döblin theorem to the function $H_{y^\ast}'(x)$ that
\begin{align*}
H_{y^\ast}'(x) &= \E_x\left[e^{-\int_0^{\tau_n}\rho(\hat{X}_s)ds}H_{y^\ast}'(\hat{X}_{\tau_n})-\int_0^{\tau_n}e^{-\int_0^{t}\rho(\hat{X}_s)ds}
(\mathcal{G}_rH_{y^\ast})'(\hat{X}_t)dt\right]\\
&\geq \E_x\left[e^{-\int_0^{\tau_n}\rho(\hat{X}_s)ds}
A'(\hat{X}_{\tau_n})\right],
\end{align*}
where $\tau_n$ is a sequence of almost surely finite stopping times converging to an arbitrary stopping time $\tau\wedge \hat{\tau}_0$ as $n\uparrow\infty$.
Invoking monotone convergence now implies that
$$
H_{y^\ast}'(x) \geq \E_x\left[e^{-\int_0^{\tau\wedge \hat{\tau}_0}\rho(\hat{X}_s)ds}
A'(\hat{X}_{\tau\wedge \hat{\tau}_0})\right].
$$
Since this inequality is valid for any stopping time, it has to be valid for the optimal one from which the alleged inequality follows.

In order to prove the second part of our lemma, it is sufficient to show that the value $H_{y^\ast}'(x)$ can be attained by following an admissible stopping time.
To see that this is indeed the case, we notice
since $\psi'(x)$ and $\varphi'(x)$ satisfy the ordinary differential equation
\begin{align}
\frac{1}{2}\sigma^2(x)v''(x)+(\mu(x)+\sigma(x)\sigma'(x))v'(x)-\rho(x)v(x) = 0\label{ode2}
\end{align}
and $\varphi''(x)\psi'(x)-\varphi'(x)\psi''(x)=2rB\hat{S}'(x)\neq 0$ all solutions of \eqref{ode2} have to be of the form $v(x)=c_1\psi'(x)+c_2\varphi'(x)$. On the other hand, as was shown in Lemma 3.3 and Lemma 3.4 of \cite{Al04}, Assumption \ref{oletukset}(B) and the assumed boundary behavior of the underlying diffusion guarantee that the fundamental solutions  $\psi(x)$ and $\varphi(x)$ are strictly convex on $\mathcal{I}$. Hence, we find that for all $x\in (z,y)$
$$
\E_x\left[e^{-\int_0^{\tilde{\tau}}\rho(\hat{X}_s)ds}\right] =\frac{v_1(x)}{v_1(z)}+\frac{v_2(x)}{v_2(y)},
$$
where $\tilde{\tau}=\inf\{t\geq 0: \hat{X}_t\not \in (z,y)\}$ denotes the first exit time of $\hat{X}$ from $(z,y)$, $v_1(x) = \frac{\varphi'(y)}{\psi'(y)}\psi'(x)-\varphi'(x)$ is the decreasing and $v_2(x)=\psi'(x)-\frac{\psi'(z)}{\varphi'(z)}\varphi'(x)$ the increasing fundamental solutions of \eqref{ode2} for $\hat{X}$ killed at the boundaries $z,y$. The assumed boundary behavior implies that $\lim_{z\downarrow 0}\psi'(z)/\varphi'(z)=0$ and $\lim_{y\rightarrow b}\varphi'(y)/\psi'(y)=0$. Hence, we notice by invoking the condition $\lim_{z\downarrow 0}\varphi'(z)=-\infty$ that
$$
\lim_{z\downarrow 0}\E_x\left[e^{-\int_0^{\tilde{\tau}}\rho(\hat{X}_s)ds}\right] = \frac{\psi'(x)}{\psi'(y)}+\lim_{z\downarrow 0}\frac{\frac{\varphi'(y)}{\psi'(y)}\psi'(x)-\varphi'(x)}{\frac{\varphi'(y)}{\psi'(y)}\psi'(z)-\varphi'(z)}=\frac{\psi'(x)}{\psi'(y)}.
$$
In a completely analogous fashion, we observe that
$$
\lim_{y\uparrow b}\E_x\left[e^{-\int_0^{\tilde{\tau}}\rho(\hat{X}_s)ds}\right] = \frac{\varphi'(x)}{\varphi'(z)}.
$$
Consequently, we notice that $\psi'(x)$ constitutes the increasing and $-\varphi'(x)$ the decreasing fundamental solutions of \eqref{ode2} for
$\hat{X}$ killed at the rate $r-\mu'(x)$ and, therefore, we have
$$
H_{y^\ast}'(x)=\E_x\left[e^{-\int_0^{\hat{\tau}_{y^\ast}\wedge \hat{\tau}_0}\rho(\hat{X}_s)ds}
A'(\hat{X}_{\hat{\tau}_{y^\ast}\wedge \hat{\tau}_0})\right],
$$
where $\hat{\tau}_{y^\ast}=\inf\{t\geq 0: \hat{X}_t\geq y^\ast\}$.
\end{proof}
Lemma \ref{apulemma} demonstrates that the marginal value of the optimal singular policy dominates the value of the associated stopping problem. Lemma \ref{apulemma} also characterizes in terms of the boundary behavior of the underlying diffusion and the limiting behavior of the derivatives of the fundamental solutions circumstances under which the the marginal value of the optimal singular policy coincides with the value of the associated stopping problem. It is worth noticing that the verification of these conditions is a relatively straightforward task in the cases where $b=\infty$ and the drift and volatility coefficients admit the series representation $\mu(x)x/\sigma^2(x)=\sum_{k=0}^{\infty}\lambda_kx^k$ and $x^2/\sigma^2(x)=\sum_{k=0}^{\infty}\delta_kx^k$ since in that case the solutions of $(\mathcal{G}_ru)(x)=0$ take the form (a {\em Frobenius series} representation, cf. \cite{Zwillinger97} pp. 364--370)
$$
u(x)=C_1x^{\eta_{+}}\sum_{k=0}^{\infty}c_kx^k+C_2x^{\eta_{-}}\sum_{k=0}^{\infty}\hat{c}_kx^k
$$
where
$$
\eta_{+}=\frac{1}{2}-\lambda_0+\sqrt{\left(\frac{1}{2}-\lambda_0\right)^2+2r\delta_0}
$$
and
$$
\eta_{-}=\frac{1}{2}-\lambda_0-\sqrt{\left(\frac{1}{2}-\lambda_0\right)^2+2r\delta_0}
$$
constitute the roots of the indicial equation $\eta(\eta-1)+2\lambda_0\eta-r\delta_0=0$. It is clear that if $r\delta_0>\lambda_0$, then $\eta_+>1$ and, consequently, $x^{\eta_{+}}$ is increasing, strictly convex, and satisfies the limiting condition $\eta_{+}x^{\eta_{+}-1}\rightarrow\infty$ as $x\rightarrow\infty$.

Given our observations on the marginal value, we now ask under the conditions of Lemma \ref{apulemma} the following question:
\begin{problem}
Is there a nondecreasing representing  function $\tilde{f}:\mathcal{I}\mapsto \mathbb{R}_+$ for which
\begin{align}\label{secondrep}
\E_x\left[\sup\{\tilde{f}(\hat{X}_t)\mathbbm{1}_{[y^\ast,b)}(\hat{X}_t);t<\hat{T}_{\rho}\}\right]=\hat{H}(x)=H_{y^\ast}'(x),
\end{align}
where $\hat{T}_{\rho}$ denotes the life time of the process $\hat{X}_t$ killed at the rate $r-\mu'(x)$. 
\end{problem}
In light of the results of Lemma \ref{apulemma} it is clear that if \eqref{secondrep} holds, then we would necessarily have
\begin{align*}%\label{family4}
\psi'(x)\int_{x\lor y^\ast}^b\tilde{f}(z)\frac{\psi''(z)}{\psi'^2(z)}dz=H_{y^\ast}'(x).
\end{align*}
Applying \eqref{family1} for $x\geq y^\ast$ then results into the identity
\begin{align*}%\label{family4}
\psi'(x)\int_{x}^b\tilde{f}(z)\frac{\psi''(z)}{\psi'^2(z)}dz=A'(x).
\end{align*}
from which we obtain by straightforward differentiation that (cf. \cite{AlMa15})
$$
\tilde{f}(x)=A'(x)-\frac{A''(x)}{\psi''(x)}\psi'(x).
$$

To derive an alternative expression for the representing function $\tilde{f}$ in terms of the representing function $f$, we notice by invoking \eqref{esitys1} that
\begin{align*}%\label{esitys2}
H_{y^\ast}'(x) = \frac{d}{dx}\left[\psi(x)\int_{x\lor y^\ast}^bf(z)\frac{\psi'(z)}{\psi^2(z)}dz\right].
\end{align*}
Consequently, the functions $\tilde{f}$ and $f$ are related with each other through the identity
$$
\psi'(x)\int_{x\lor y^\ast}^b\tilde{f}(z)\frac{\psi''(z)}{\psi'^2(z)}dz = \frac{d}{dx}\left[\psi(x)\int_{x\lor y^\ast}^bf(z)\frac{\psi'(z)}{\psi^2(z)}dz\right].
$$
For $x\in[y^\ast,b)$ we have that
$$
\int_{x}^b\tilde{f}(z)\frac{\psi''(z)}{\psi'^2(z)}dz=\int_{x}^bf(z)\frac{\psi'(z)}{\psi^2(z)}dz-\frac{f(x)}{\psi(x)}.
$$
Standard differentiation then yields
$$
\tilde{f}(x)=\frac{2rS'(x)}{\sigma^2(x)\psi''(x)}\int_0^x\psi(t)((\mathcal{G}_rA)(t)-(\mathcal{G}_rA)(x))m'(t)dt=
-\frac{2S'(x)}{\sigma^2(x)\psi''(x)}\int_{0}^x\frac{\psi'(t)}{S'(t)}(\mathcal{G}_rA)'(t)dt.
$$
Given this expression, we are now in position to establish the following representation result (cf. Proposition 2.13 in \cite{ChSaTa} and Theorem 3.7 in \cite{AlMa15}).
\begin{theorem}\label{apulause2}
Assume that the conditions of Theorem \ref{singesitys2} and Assumptions \ref{oletukset} are satisfied. Assume also that $0$ is unattainable for $X$ and $\hat{X}$ killed at the rate $\rho(x)$ and that the function
$(\mathcal{G}_rA)'(x)/\rho(x)$
is non-increasing on $\mathcal{I}$. Then
\begin{align}\label{rep3}
\tilde{f}(x)=A'(x)-\frac{A''(x)}{\psi''(x)}\psi'(x)=-\frac{\int_{0}^x\psi'(t)(\mathcal{G}_rA)'(t)\hat{m}'(t)dt}{\int_0^x\rho(t)\psi'(t)\hat{m}'(t)dt}
\end{align}
is nondreasing on $[y^\ast,b)$ and satisfies the condition $\tilde{f}(y^\ast)=0$. Moreover,
\begin{align}\label{rep4}
\hat{H}(x)=H_{y^\ast}'(x)= \E_x\left[\sup\left\{-\frac{\int_{y^\ast}^{\hat{X}_t}\psi'(t)(\mathcal{G}_rA)'(t)\hat{m}'(t)dt}
{\int_0^{\hat{X}_t}\rho(t)\psi'(t)\hat{m}'(t)dt}\mathbbm{1}_{[y^\ast,b)}(\hat{X}_t);t<\hat{T}_{\rho}\right\}\right].
\end{align}
\end{theorem}
\begin{proof}
We first observe that $\psi(x)$ has nice second order properties in the sense that
$$
\frac{d}{dx}\left[\frac{1}{2}\sigma^2(x)\frac{\psi''(x)}{S'(x)}\right]= \rho(x)\frac{\psi'(x)}{S'(x)}>0
$$
for all $x\in \mathcal{I}$. Since $\psi'(x)$ constitutes the increasing fundamental solution of \eqref{ode2} for $\hat{X}$  killed at the rate $\rho(x)$, we notice that the assumed boundary behavior guarantees that $\psi''(x)/\hat{S}'(x)\rightarrow 0$ as $x\downarrow 0$. Hence
$$
\frac{1}{2}\sigma^2(x)\frac{\psi''(x)}{S'(x)} = \int_0^x\rho(t)\frac{\psi'(t)}{S'(t)}dt,
$$
which proves that $\tilde{f}$ has the alleged form \eqref{rep3} under the assumptions of our theorem. It is now sufficient to prove that $\tilde{f}$ is nondecreasing. To see that this is indeed the case, we notice by ordinary differentiation that
\begin{align*}
\tilde{f}'(x)=\frac{\rho(x)\psi'(x)\hat{m}'(x)}{\left(\int_0^x\rho(t)\psi'(t)\hat{m}'(t)dt\right)^2}
\left(\int_{0}^x\rho(t)\psi'(t)\left(\frac{(\mathcal{G}_rA)'(t)}{\rho(t)}-\frac{(\mathcal{G}_rA)'(x)}{\rho(x)}\right)\hat{m}'(t)dt\right)\geq 0
\end{align*}
since $(\mathcal{G}_rA)'(x)/\rho(x)$ was assumed to be non-increasing. Hence, $\tilde{f}$ is nondecreasing on $\mathcal{I}$. Finally, representation
\eqref{rep4} follows from \eqref{secondrep}.
\end{proof}
It is worth noticing that our results actually prove that
\begin{align*}
f'(x)=\tilde{f}(x)\frac{\psi''(x)}{\psi'(x)^2}\psi(x) = \tilde{f}(x)\frac{d\ln\psi'(x)}{d\ln\psi(x)}.
\end{align*}
Hence, applying Proposition \ref{prop y} shows that $\tilde{f}$ does not only dictate the behavior of the value of the stopping problem (cf. Theorem 3.7 in \cite{AlMa15}), it also determines the behavior of the value of the singular control problem.

Under specific circumstances our results are also directly related with the optimal signal results related to Gittins indices (cf. \cite{BaBa,ElKarKar94,Kar1984}). In \cite{BaBa} it was shown that the stopping set of an optimal stopping problem equals to a set where the so called optimal stopping signal $\gamma:\mathcal{I}\to\R$ associated with the particular stopping problem is non-negative. This has been further connected to our representation in \cite[Proposition 2.14]{AlMa15}, where it was proved that the representing function $\tilde{f}(x)$ actually coincides with $\gamma(x)$ on the stopping set of an associated nonstandard optimal stopping problem.

To see how our present representation is related to the optimal stopping signal, we first denote as $\mathcal{T}$ the set of first exit times $\inf\{t\geq 0: \hat{X}_t\not\in (z,y)\}$ from open intervals $(z,y)\subset \mathcal{I}$ with compact closure in $\mathcal{I}$. Our main finding on the connection between the developed approach and optimal signals is summarized in the following.
\begin{theorem}\label{apulause3}
Assume that the conditions of Theorem \ref{singesitys2} and Assumptions \ref{oletukset} are satisfied. Assume also that $0$ is unattainable for $X$ and $\hat{X}$ killed at the rate $\rho(x)$ and that the flow $\pi(x)$ is monotonically increasing on $\mathcal{I}$. Then $x\geq y^\ast$ if and only if $\hat{\gamma}(x)\geq 1$, where
\begin{align*}
\hat{\gamma}(x) = \inf_{\hat{\tau}\in \mathcal{T}}\frac{\E_x\left[\alpha(x)-e^{-\int_0^{\hat{\tau}}\rho(\hat{X}_s)ds}\alpha(\hat{X}_{\hat{\tau}})\right]}{\E_x\left[(R_r\pi)'(x)-
e^{-\int_0^{\hat{\tau}}\rho(\hat{X}_s)ds}(R_r\pi)'(\hat{X}_{\hat{\tau}})\right]}.
\end{align*}
If the flow $\pi(x)$ is monotonically decreasing on $\mathcal{I}$, then $x\geq y^\ast$ if and only if $\check{\gamma}(x)\leq 1$, where
\begin{align*}
\check{\gamma}(x) = \sup_{\hat{\tau}\in \mathcal{T}}\frac{\E_x\left[\alpha(x)-e^{-\int_0^{\hat{\tau}}\rho(\hat{X}_s)ds}\alpha(\hat{X}_{\hat{\tau}})\right]}{\E_x\left[(R_r\pi)'(x)-
e^{-\int_0^{\hat{\tau}}\rho(\hat{X}_s)ds}(R_r\pi)'(\hat{X}_{\hat{\tau}})\right]}
\end{align*}
\end{theorem}
\begin{proof}
We first notice that the assumed monotonicity of the flow $\pi(x)$ implies that the expected cumulative present value $(R_r\pi)(x)$ is monotonically increasing as well and, therefore, that $(R_r\pi)'(x)>0$ for all $x\in \mathcal{I}$. On the other hand, since $(R_r\pi)(x)$ satisfies the ordinary differential equation $(\mathcal{G}_r R_r\pi)(x)+\pi(x)=0$ and $\pi(x)$ was assumed to be continuously differentiable, we notice by differentiating that
\begin{align*}
\frac{1}{2}\sigma^2(x)(R_r\pi)'''(x) + (\mu(x)+\sigma(x)\sigma'(x))(R_r\pi)''(x)-\rho(x)(R_r\pi)'(x) =-\pi'(x)<0
\end{align*}
for all $x\in \mathcal{I}$. Consequently, if $\hat{\tau}\in \mathcal{T}$, then
\begin{align*}
\E_x\left[e^{-\int_0^{\hat{\tau}}\rho(\hat{X}_s)ds}(R_r\pi)'(\hat{X}_{\hat{\tau}})\right]=(R_r\pi)'(x)-\E_x\int_0^{\hat{\tau}}e^{-\int_0^{t}\rho(\hat{X}_s)ds}\pi'(\hat{X}_t)dt\leq (R_r\pi)'(x)
\end{align*}
for all $x\in \mathcal{I}$.

It is clear that under the conditions of our theorem $H_{y^\ast}'(x)$ constitutes the smallest excessive majorant of the exercise payoff $A'(x)=\alpha(x)-(R_r\pi)'(x)$ for the diffusion $\hat{X}_t$ killed at the rate $\rho(x)$. However, a nonnegative and measurable function $l:\mathcal{I}\mapsto\mathbb{R}_+\cup\{\infty\}$
is excessive for the diffusion
$\hat{X}_t$ killed at the rate $\rho(x)$ if, and only if, $l(x)$ is continuous and satisfies for all $x\in \mathcal{I}$ the inequality
\begin{align*}
\mathbb{E}_x\left[e^{-\int_0^{\hat{\tau}}\rho(\hat{X}_s)ds}l(\hat{X}_{\hat{\tau}})\right]\leq l(x),
\end{align*}
where  $\hat{\tau}\in \mathcal{T}$ is an arbitrary first exit time from an arbitrary
open set with compact closure in $\mathcal{I}$ (cf. \cite{Dynkin}, Theorem 12.4. on pp. 7--8). Consequently, we find that
$x\in \{x\in \mathcal{I}: H_{y^\ast}'(x)=\alpha(x)-(R_r\pi)'(x)\}=[y^\ast,b)$ if, and only if
\begin{align*}
\E_x\left[e^{-\int_0^{\hat{\tau}}\rho(\hat{X}_s)ds}(\alpha(\hat{X}_{\hat{\tau}})-(R_r\pi)'(\hat{X}_{\hat{\tau}}))\right]\leq \alpha(x)-(R_r\pi)'(x)
\end{align*}
for all $\hat{\tau}\in \mathcal{T}$. Reordering terms then yields that we have
\begin{align*}
\frac{\E_x\left[\alpha(x)-e^{-\int_0^{\hat{\tau}}\rho(\hat{X}_s)ds}\alpha(\hat{X}_{\hat{\tau}})\right]}{\E_x\left[(R_r\pi)'(x)-
e^{-\int_0^{\hat{\tau}}\rho(\hat{X}_s)ds}(R_r\pi)'(\hat{X}_{\hat{\tau}})\right]}\geq 1
\end{align*}
for all $\hat{\tau}\in \mathcal{T}$. On the other hand, if $x\in \{x\in \mathcal{I}: H_{y^\ast}'(x)>\alpha(x)-(R_r\pi)'(x)\}=(a,y^\ast)$, then
\begin{align*}
\E_x\left[e^{-\int_0^{\hat{\tau}^\ast}\rho(\hat{X}_s)ds}(\alpha(\hat{X}_{\hat{\tau}^\ast})-(R_r\pi)'(\hat{X}_{\hat{\tau}^\ast}))\right] > \alpha(x)-(R_r\pi)'(x),
\end{align*}
where $\hat{\tau}^\ast=\inf\{t\geq 0:\hat{X}_t\geq y^\ast\}$ denotes the optimal exercise time. Hence, we find that $[y^\ast,b)=\{x\in \mathcal{I}: \hat{\gamma}(x)\geq 1\}$ as claimed. Establishing the signal interpretation in the monotonically decreasing setting is completely analogous.
\end{proof}

\section{\bf Explicit Example} Our objective is now to illustrate the considered class of singular stochastic control problems and the possible intricacies associated with the developed representation results in an explicitly parameterized example. To this end, we assume that the underlying dynamics evolve in the absence of interventions according to a standard GBM characterized by the stochastic differential equation \eqref{singcontsde} with a drift $\mu(x)=\tilde{\mu} x$ and a volatility coefficient $\sigma(x) = \tilde{\sigma} x$, where $\tilde{\sigma}>0$. We also assume that $\pi(x)=x^\eta$ and $\alpha(x)=K_1 e^{-\nu x} + K_2(1-e^{-\nu x})$, where $\eta\in(0,1)$, $K_1>K_2>0$, and $\nu>0$ are exogenously determined constants.

It is well-known from the literature that in this case the fundamental solutions read as $\psi(x)=x^{\kappa}$ and $\varphi(x)=x^{\vartheta}$,  where
$$
\kappa=\frac{1}{2}-\frac{\tilde{\mu}}{\tilde{\sigma}^2} + \sqrt{\left(\frac{1}{2}-\frac{\tilde{\mu}}{\tilde{\sigma}^2}\right)^2+\frac{2r}{\tilde{\sigma}^2}}>0
$$
and
$$
\vartheta=\frac{1}{2}-\frac{\tilde{\mu}}{\tilde{\sigma}^2} - \sqrt{\left(\frac{1}{2}-\frac{\tilde{\mu}}{\tilde{\sigma}^2}\right)^2+\frac{2r}{\tilde{\sigma}^2}}<0.
$$
Moreover, $(R_r\pi)(x)=Mx^\eta$, where
\begin{align}\label{multiplier}
M=\frac{1}{r-\mu\eta-\frac{1}{2}\sigma^2\eta(\eta-1)}>0.
\end{align}

Utilizing these findings show that if $\kappa >1$ (i.e. if $r>\tilde{\mu}$), then
$$
\frac{A'(x)}{\psi'(x)}=\frac{K_1 e^{-\nu x} + K_2(1-e^{-\nu x})-\eta M x^{\eta-1}}{\kappa x^{\kappa-1}}\in \left(\frac{K_2-\eta M x^{\eta-1}}{\kappa x^{\kappa-1}},\frac{K_1 -\eta M x^{\eta-1}}{\kappa x^{\kappa-1}}\right)
$$
and $A'(x)/\psi'(x)$ attains a global maximum at a uniquely defined threshold
$$
y^\ast \in \left(\left(\frac{M\eta(\kappa-\eta)}{(\kappa-1)K_1}\right)^{\frac{1}{1-\eta}},\left(\frac{M\eta(\kappa-\eta)}{(\kappa-1)K_2}\right)^{\frac{1}{1-\eta}}\right)
$$
satisfying the ordinary first order condition
$$
M\eta(\kappa-\eta){y^\ast}^{\eta-1}-(K_1-K_2)e^{-\nu y^\ast}(\kappa-1+\nu y^\ast)-\nu(K_1-K_2)e^{-\nu y^\ast}=0.
$$
In this case, the value of the optimal policy reads as
$$
V(x)=\begin{cases}
\frac{K_1-K_2}{\nu}\left(e^{-\nu y^\ast}-e^{-\nu x}\right)+K_2(x-y^\ast) + M{y^\ast}^\eta +\frac{A'(y^\ast)}{\kappa}{y^\ast}&x\geq y^\ast\\
Mx^\eta + \frac{A'(y^\ast)}{\kappa}x^\kappa {y^\ast}^{1-\kappa} & x <y^\ast
\end{cases}
$$
and the function needed for the representation as an expected supremum reads as in \eqref{incsing} with
$$
A(x)-\frac{A'(x)}{\psi'(x)}\psi(x)=\frac{(\kappa-1)K_2x}{\kappa}-\frac{\kappa-\eta}{\kappa}Mx^\eta+(K_1-K_2)\left(\frac{1-e^{-\nu x}}{\nu}-\frac{e^{-\nu x}x}{\kappa}\right).
$$
The representing functions $f$ associated with three different thresholds (the optimal one as well as two suboptimal) are illustrated in Figure \ref{kuva1} in the case where
$\sigma = 0.1, r = 0.05, \mu = 0.01, K_1 = 12, K_2 = 10, \nu = 0.1$, and
$\eta = 0.5$. As is clear from Figure \ref{kuva1}, choosing the threshold $y^\ast$ optimally results into a representing function which is nonnegative on the entire state space. As soon as the threshold is chosen in a suboptimal fashion, the representing function can attain also negative values.
\begin{figure}[h!]
\begin{center}
\includegraphics[width=0.5\linewidth]{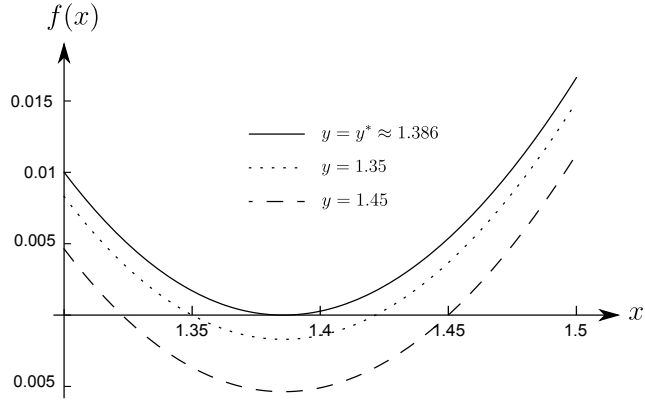}
\end{center}
\caption{\small{The representing function $f$.}}\label{kuva1}
\end{figure}

We also notice that if $\kappa>1$ then
$$
\hat{H}(x)=H_{y^\ast}'(x)=\begin{cases}
K_1e^{-\nu x} + (1-e^{-\nu x})K_2 - \eta M x^{\eta-1}&x\geq y^\ast\\
(K_1e^{-\nu y^\ast} + (1-e^{-\nu y^\ast})K_2-\eta M {y^\ast}^{\eta-1})\left(\frac{x}{y^\ast}\right)^{\kappa-1}&x<y^\ast
\end{cases}
$$
and
$$
\tilde{f}(x)=\frac{1}{\kappa-1}\left((K_1-K_2)e^{-\nu x}(\nu x +\kappa-1)-(\kappa-\eta)\eta M x^{\eta-1}+K_2(\kappa-1)\right).
$$

\section{Conclusions}

We considered a class of singular stochastic control problems admitting the representation of both the value of the optimal policy as well as its marginal value as an expected supremum. Instead of determining the optimal policy and its value at once, we first stated the expected cumulative present value of the revenue flow accrued by following an ordinary reflection policy at an arbitrary constant boundary. We then computed the expected value of an unknown function depending on the running supremum attained at an exponentially distributed random time by utilizing the known probability distribution of the running supremum. Setting these values equal to each other then results into a functional equation from which the unknown function is then determined.

In contrast with optimal stopping problems, the function needed for the representation vanishes at the reflection boundary independently of wether it is optimal or not. As our results indicate, optimality is attained at the threshold where the derivative of the above mentioned function vanishes. This condition is shown to coincide with the standard $C^2$-smoothness condition arising in singular stochastic control problems. Interestingly, since the marginal value of the singular stochastic control problem constitutes under some circumstances the value of an associated stopping problem, a similar supremum representation was shown to hold for the marginal value as well. Since the functions needed for the representations depend on each other, it is precisely this connection which explains the different regularity conditions needed for the characterization of the optimal boundary.

This study focused on a single boundary setting where the reflection policy is exerted only at a single optimal state resulting into a maximized expected cumulative yield.  Even though this is sufficient from the perspective of many financial and economic applications of singular stochastic control theory, it would be mathematically interesting to try to investigate if a similar representation as an expected running supremum could be determined in a case where the underlying can be controlled both upwards as well as downwards within the considered reflecting policy setting (cf. \cite{Shreve84} and \cite{Ma12}). It is clear in light of previous results on two-boundary optimal stopping problems that such a problem formulation typically result into the analysis of quite demanding nontrivial integral equations. A second natural direction towards which our approach could be developed is impulse control (cf. \cite{ChSa15}). Since optimal reflection can be typically seen as a limiting policy arising as the fixed costs associated with the implementation of an impulse policy tend to zero, it would be of interest to see how this limiting behavior arises within an expected supremum setting. Especially, it would be of interest to see how the extreme sensitivity of the value with respect to small changes in costs (cf. \cite{Ok99}) affect the representation. Both these extensions are out of the scope of the present study and left for the future.

\bibliographystyle{amsplain}
\bibliography{Maks}

\providecommand{\bysame}{\leavevmode\hbox to3em{\hrulefill}\thinspace}
\providecommand{\MR}{\relax\ifhmode\unskip\space\fi MR }
\providecommand{\MRhref}[2]{%
  \href{http://www.ams.org/mathscinet-getitem?mr=#1}{#2}
}
\providecommand{\href}[2]{#2}
\begin{thebibliography}{10}

\bibitem{AkMeSu96}
M.~Akian, J.~L. Menaldi, and A.~Sulem, \emph{On an investment-consumption model
  with transaction costs}, SIAM J. Control Optim. \textbf{34} (1996), no.~1,
  329--364.

\bibitem{Al00}
L.~H.~R. Alvarez~E., \emph{Singular stochastic control in the presence of a
  state-dependent yield structure}, Stochastic Process. Appl. \textbf{86}
  (2000), no.~2, 323--343.

\bibitem{Al04}
\bysame, \emph{A class of solvable stopping games}, Appl. Math. Optim.
  \textbf{58} (2008), no.~3, 291--314.

\bibitem{Al2011}
\bysame, \emph{Optimal capital accumulation under price uncertainty and costly
  reversibility}, J. Econom. Dynam. Control \textbf{35} (2011), no.~10,
  1769--1788.

\bibitem{AlSh98}
L.~H.~R. Alvarez~E. and L.~A. Shepp, \emph{Optimal harvesting of stochastically
  fluctuating populations}, J. Math. Biol. \textbf{37} (1998), no.~2, 155--177.

\bibitem{AlVi}
L.~H.~R. Alvarez~E. and J.~Virtanen, \emph{A class of solvable stochastic
  dividend optimization problems: on the general impact of flexibility on
  valuation}, Econom. Theory \textbf{28} (2006), no.~2, 373--398.

\bibitem{AlMa15}
L.H.R. Alvarez~E. and P.~Matomäki, \emph{Expected supremum representation and
  optimal stopping}, arXiv:1505.01660 (2015).

\bibitem{Ba}
F.~M. Baldursson, \emph{Singular stochastic control and optimal stopping},
  Stochastics \textbf{21} (1987), no.~1, 1--40.

\bibitem{BaKa}
F.~M. Baldursson and I.~Karatzas, \emph{Irreversible investment and industry
  equilibrium}, Financ. Stoch. \textbf{1} (1997), 69 -- 89.

\bibitem{BaBa}
P.~Bank and C.~Baumgarten, \emph{Parameter-dependent optimal stopping problems
  for one-dimensional diffusions}, Electron. J. Probab. \textbf{15} (2010),
  1971–1993.

\bibitem{BaElKa}
P.~Bank and N.~El~Karoui, \emph{A stochastic representation theorem with
  applications to optimization and obstacle problems}, Ann. Probab. \textbf{32}
  (2004), 1030–1067.

\bibitem{BaFo}
P.~Bank and H.~F{\"o}llmer, \emph{American options, multi-armed bandits, and
  optimal consumption plans: a unifying view}, Paris-{P}rinceton {L}ectures on
  {M}athematical {F}inance, 2002, Lecture Notes in Math., vol. 1814, Springer,
  Berlin, 2003, pp.~1--42.

\bibitem{BaRi}
P.~Bank and F.~Riedel, \emph{Optimal consumption choice with intertemporal
  substitution}, Ann. Appl. Probab. \textbf{11} (2001), 750--788.

\bibitem{BSW}
V.~E. Bene{\v{s}}, L.~A. Shepp, and H.~S. Witsenhausen, \emph{Some solvable
  stochastic control problems}, Stochastics \textbf{4} (1980/81), no.~1,
  39--83.

\bibitem{BK}
F.~Boetius and M.~Kohlmann, \emph{Connections between optimal stopping and
  singular stochastic control}, Stochastic Process. Appl. \textbf{77} (1998),
  no.~2, 253--281.

\bibitem{BorSal02}
A.~Borodin and P.~Salminen, \emph{Handbook of brownian motion - facts and
  formulae}, Birkhauser, Basel, 2002.

\bibitem{CaSaZa}
A.~Cadenillas, S.~Sarkar, and F.~Zapatero, \emph{Optimal dividend policy with
  mean-reverting cash reservoir}, Math. Finance \textbf{17} (2007), no.~1,
  81--109.

\bibitem{ChMeRo}
P.~L. Chow, J.~L. Menaldi, and M.~Robin, \emph{Additive control of stochastic
  linear systems with finite horizon}, SIAM J. Control Optim. \textbf{23}
  (1985), no.~6, 858--899.

\bibitem{ChSa15}
S.~Christensen and P.~Salminen, \emph{Impulse control and expected suprema},
  arXiv:1503.01253 (2015).

\bibitem{ChSaTa}
S.~Christensen, P.~Salminen, and B.~Q. Ta, \emph{Optimal stopping of strong
  markov processes}, Stochastic Process. Appl. \textbf{123} (2013), 1138–1159.

\bibitem{DeFeMo15}
T.~De~Angelis, G.~Ferrari, and J.~Moriarty, \emph{A {N}onconvex {S}ingular
  {S}tochastic {C}ontrol {P}roblem and its {R}elated {O}ptimal {S}topping
  {B}oundaries}, SIAM J. Control Optim. \textbf{53} (2015), no.~3, 1199--1223.

\bibitem{Dixit1994}
A.~K. Dixit and R.~S. Pindyck, \emph{Investment under uncertainty}, Princeton
  University Press, New Jersey, 1994.

\bibitem{Dynkin}
E.~B. Dynkin, \emph{Markov processes: volume ii}, Academic Press Inc.,
  Publishers, New York; Springer-Verlag, Berlin-G\"ottingen-Heidelberg, 1965.

\bibitem{ElKaFo}
N.~El~Karoui and H.~Föllmer, \emph{A non-linear riesz respresentation in
  probabilistic potential theory}, Ann. Inst. H. Poincaré Probab. Statist.
  \textbf{41} (2005), 269--283.

\bibitem{ElKaKa91}
N.~El~Karoui and I.~Karatzas, \emph{{A} new approach to the {S}korohod problem,
  and its applications}, Stochastics Stochastics Rep. \textbf{34} (1991),
  no.~1-2, 57–82.

\bibitem{ElKarKar94}
\bysame, \emph{Dynamic allocation problems in continuous time}, Ann. Appl.
  Probab. \textbf{4} (1994), no.~2, 255--286.

\bibitem{ElKaMe}
N.~El~Karoui and A.~Meziou, \emph{Max–plus decomposition of supermartingales
  and convex order. application to american options and portfolio insurance},
  Ann. Probab. \textbf{36} (2008), 647--697.

\bibitem{FoKn1}
H.~Föllmer and T.~Knispel, \emph{A representation of excessive functions as
  expected suprema}, Probab. Math. Statist. \textbf{26} (2006), 379–394.

\bibitem{FoKn2}
\bysame, \emph{Potentials of a markov process are expected suprema}, ESAIM:
  Probabability and Statistics \textbf{11} (2007), 89–101.

\bibitem{HaSu1}
U.~G. Haussmann and W.~Suo, \emph{Singular optimal stochastic controls {I}:
  {E}xistence}, SIAM J. Control Optim. \textbf{33} (1995), no.~3, 916--936.

\bibitem{HaSu2}
\bysame, \emph{Singular optimal stochastic controls {II}: {D}ynamic
  programming}, SIAM J. Control Optim. \textbf{33} (1995), no.~3, 937--959.

\bibitem{JS}
M.~Jeanblanc-Picqu{\'e} and A.~N. Shiryaev, \emph{Optimization of the flow of
  dividends}, Russian Math. Surveys \textbf{50} (1995), 257 -- 277.

\bibitem{Kar1}
I.~Karatzas, \emph{A class of singular stochastic control problems}, Adv. in
  Appl. Probab. \textbf{15} (1983), no.~2, 225--254.

\bibitem{Kar1984}
\bysame, \emph{Gittins indices in the dynamic allocation problem for diffusion
  processes}, Ann. Probab. \textbf{12} (1984), no.~1, 173--192.

\bibitem{KS1}
I.~Karatzas and S.~E. Shreve, \emph{Connections between optimal stopping and
  singular stochastic control {I}. {M}onotone follower problems}, SIAM J.
  Control Optim. \textbf{22} (1984), no.~6, 856--877.

\bibitem{Karatzas88}
I.~Karatzas and S.~E. Shreve, \emph{Brownian motion and stochastic calculus},
  Springer-Verlag New York, 1988.

\bibitem{Ko}
T.~{\O}. Kobila, \emph{A class of solvable stochastic investment problems
  involving singular controls}, Stochastics Stochastics Rep. \textbf{43}
  (1993), no.~1-2, 29--63.

\bibitem{LES1}
R.~Lande, Engen S., and S{\ae}ther B.-E., \emph{Optimal harvesting, economic
  discounting and extinction risk influctuating populations}, Nature
  \textbf{372} (1994), 88 -- 90.

\bibitem{LES2}
\bysame, \emph{Optimal harvesting of fluctuating populations with a risk of
  extinction}, Am. Nat. \textbf{145} (1995), 728 -- 745.

\bibitem{LO}
E.~M. Lungu and B.~{\O}ksendal, \emph{Optimal harvesting from a population in a
  stochastic crowded environment}, Math. Biosci \textbf{145} (1997), no.~1,
  47--75.

\bibitem{Ma12}
P.~Matom{\"a}ki, \emph{On solvability of a two-sided singular control problem},
  Math. Methods Oper. Res. \textbf{76} (2012), no.~3, 239--271.

\bibitem{AOks2000}
A.~{\O}ksendal, \emph{Irreversible investment problems}, Financ. Stoch.
  \textbf{4} (2000), 223--250.

\bibitem{Ok99}
B.~{\O}ksendal, \emph{Stochastic control problems where small intervention
  costs have big effects}, Appl. Math. Optim. \textbf{40} (1999), no.~3,
  355--375.

\bibitem{Protter05}
Philip~E. Protter, \emph{Stochastic integration and differential equations},
  Stochastic Modelling and Applied Probability, vol.~21, Springer-Verlag,
  Berlin, 2005, Second edition. Version 2.1, Corrected third printing.

\bibitem{RogWil00}
L.~C.~G. Rogers and D.~Williams, \emph{Diffusions, {M}arkov processes, and
  martingales. {V}ol. 1}, Cambridge Mathematical Library, Cambridge University
  Press, Cambridge, 2000, Foundations, Reprint of the second (1994) edition.

\bibitem{Salminen1985}
P.~Salminen, \emph{Optimal stopping of one-dimensional diffusions}, Math.
  Nachr. \textbf{124} (1985), 85--101.

\bibitem{Shreve84}
S.E. Shreve, J.P. Lehoczky, and D.P. Gaver, \emph{Optimal consumption for
  general diffusion with absorbing and reflecting barriers}, SIAM J. Control
  Optim. \textbf{22} (1984), 55--75.

\bibitem{Zwillinger97}
D.~Zwillinger, \emph{Handbook of differential equations}, third ed., Academic
  Press, Inc., Boston, MA, 1997.

\end{thebibliography}

\end{document}